\documentclass[11pt,reqno]{amsart}
\usepackage{amsmath, amsfonts, color, amsthm, graphics, amssymb,lscape}
\usepackage[notcite,notref,final]{showkeys}

\date{}

\newtheorem{thm}{Theorem}[section]
\newtheorem{cor}[thm]{Corollary}
\newtheorem{lem}[thm]{Lemma}
\newtheorem{prop}[thm]{Proposition}
\newtheorem{rmk}[thm]{Remark}

\theoremstyle{definition}
\newtheorem{defn}[thm]{Definition}

\numberwithin{equation}{section}

\newcommand{\mf}{\mathfrak}
\newcommand{\Z}{\mathbb{Z}}

\newcommand{\N}{\mathbb{N}}
\newcommand{\V}{\mathbb{V}}

\newcommand{\lm}{\lambda}
\newcommand{\ten}{\otimes}

\newcommand{\ra}{\rightarrow}
\newcommand{\lra}{\longrightarrow}

\newcommand{\ds}{\displaystyle}
\newcommand{\mc}{\mathcal}

\newcommand{\rank}{\textrm{rank}}

\newcommand{\ones}{{\bf 1}}

\input xy
\xyoption{all}

\begin{document}

\title[Degenerate Sklyanin algebras]{Degenerate Sklyanin algebras and Generalized Twisted Homogeneous Coordinate rings}
\author{Chelsea Walton}

\date{January 17, 2009}
\thanks{The author was partially supported by the NSF: grants DMS-0555750, 0502170.}


\maketitle
\vspace{-.1in}

\begin{center}
\footnotesize{{Department of Mathematics\\ University of Michigan\\ Ann Arbor, MI 48109.}}\\
\footnotesize{{\it E-mail address}: {\tt notlaw@umich.edu}}
\end{center}

\begin{abstract}
In this work, we introduce the point parameter ring $B$, a generalized twisted homogeneous coordinate ring associated to a degenerate version of the three-dimensional Sklyanin algebra. The surprising geometry of these algebras yields an analogue to a result of Artin-Tate-van den Bergh, namely that $B$ is generated in degree one and thus is a factor of the corresponding degenerate Sklyanin algebra. 
\end{abstract}


\section{Introduction} 
Let $k$ be an algebraically closed field of characteristic 0. We say a $k$-algebra $R$ is {\it connected graded (cg)} when $R = \bigoplus_{i \in \N} R_i$ is $\N$-graded with $R_0 = k$.

A vital development in the field of Noncommutative Projective Algebraic Geometry is the investigation of connected graded noncommutative rings with use of geometric data. In particular, a method was introduced by Artin-Tate-van den Bergh in \cite{ATV1} to construct corresponding well-behaved graded rings, namely twisted homogeneous coordinate rings (tcr) \cite{ASta, KeRSta, SmSta}. However, there exist noncommutative rings that do not have sufficient geometry to undergo this process \cite{KeRSta}. The purpose of this paper is to explore a recipe suggested in \cite{ATV1} for building a generalized analogue of a tcr for {\it any} connected graded ring. As a result, we provide a geometric approach to examine all degenerations of the Sklyanin algebras studied in \cite{ATV1}.

We begin with a few historical remarks. In the mid-1980s, Artin and Schelter \cite{ASc} began the task of classifying noncommutative analogues of the polynomial ring in three variables, yet the rings of interest were not well understood. How close were these noncommutative rings to the commutative counterpart $k[x,y,z]$? Were they Noetherian? Domains? Global dimension 3? These questions were answered later in \cite{ATV1} and the toughest challenge was analyzing the following class of algebras.

\begin{defn} \label{def:Skl3} Let $k\{x,y,z\}$ denote the free algebra on the noncommuting variables $x, y$, and $z$. The {\it three-dimensional Sklyanin algebras} are defined as
\begin{equation} \label{eq:S(abc)}
S(a,b,c) ~=~ \frac{k\{x,y,z\}}{\left(\begin{array}{c}ayz+ bzy+ cx^2,\\azx+ bxz+cy^2,\\axy+ byx+ cz^2 \end{array}\right)}
\end{equation}
for $[a:b:c] \in \mathbb{P}^2_k \setminus \mf{D}$ where 
$$\mf{D} = \{[0:0:1], [0:1:0], [1:0:0]\} ~\cup~ \{[a:b:c] ~|~ a^3=b^3=c^3=1\}.$$
\end{defn}

As algebraic techniques were exhausted, two seminal papers \cite{ATV1} and \cite{ATV2} arose introducing algebro-geometric methods to examine noncommutative analogues of the polynomial ring. In fact, a geometric framework was specifically associated to the Sklyanin algebras $S(a,b,c)$ via the following definition and result of \cite{ATV1}.

\begin{defn} \label{def:ptmod} A {\it point module} over a ring $R$ is a cyclic graded left $R$-module $M$ where $\dim_k{M_i}=1$ for all $i$. 
\end{defn}

\begin{thm} \label{thm:Skl3} Point modules for $S=S(a,b,c)$ with $[a:b:c] \notin \mf{D}$ are parameterized by the points of a smooth cubic curve 
\begin{equation} \label{eq:E(abc)}
E=E_{a,b,c}: (a^3+b^3+c^3)xyz - (abc)(x^3+y^3+z^3) = 0 \subset \mathbb{P}^2.
\end{equation}
\noindent The curve $E$ is equipped with $\sigma \in$ Aut($E$) and the invertible sheaf $i^{\ast}\mc{O}_{\mathbb{P}^2}(1)$ from which we form the corresponding twisted homogeneous coordinate ring $B$. There exists a regular normal element $g \in S$, homogeneous of degree 3, so that $B \cong S/gS$ as graded rings. The ring $B$ is a Noetherian domain and thus so is $S$. Moreover for $d \geq 1$, we get dim$_k B_d = 3d$. Hence $S$ has the same Hilbert series as $k[x,y,z]$, namely $H_S(t) = \frac{1}{(1-t)^3}$.
\qed
\end{thm}

In short, the tcr $B$ associated to $S(a,b,c)$ proved useful in determining the Sklyanin algebras' behavior. 
\medskip

Due to the importance of the Sklyanin algebras, it is natural to understand their degenerations to the set $\mf D$.

\begin{defn} \label{def:Sdeg} The rings $S(a,b,c)$ from (\ref{eq:S(abc)}) with $[a:b:c] \in \mf{D}$ are called the {\it degenerate three-dimensional Sklyanin algebras}. Such a ring is denoted by $S(a,b,c)$ or $S_{deg}$ for short.
\end{defn}

In section 2, we study the basic properties of degenerate Sklyanin algebras resulting in the following proposition.

\begin{prop} \label{prop:S(1bc)}
The degenerate three-dimensional Sklyanin algebras have Hilbert series $H_{S_{deg}}(t) = \frac{1+t}{1-2t}$, they have infinite Gelfand Kirillov dimension, and are not left or right Noetherian, nor are they domains. Furthermore, the algebras $S_{deg}$ are Koszul and have infinite global dimension.
\end{prop}

The remaining two sections construct a generalized twisted homogeneous coordinate ring $B=B(S_{deg})$ for the degenerate Sklyanin algebras. We are specifically interested in point modules over $S_{deg}$ (Definition \ref{def:ptmod}).
Unlike their nondegenerate counterparts, the point modules over $S_{deg}$ are \underline{not} parameterized by a projective scheme so care is required.
Nevertheless, the degenerate Sklyanin algebras \underline{do} have geometric data which is described by the following definition and theorem.

\begin{defn} \label{def:trunc} A {\it truncated point module of length d} over a ring $R$ is a cyclic graded left $R$-module $M$ where dim$_k M_i =1$ for $0 \leq i \leq d$ and dim$_k M_i =0$ for $i > d$. 
The {\it $d^{th}$ truncated point scheme} $V_d$ parameterizes isomorphism classes of length $d$ truncated point modules. 
\end{defn}

\begin{thm} \label{thm:truncptsch} For $d \geq 2$, the truncated point schemes $V_d \subset (\mathbb{P}^2)^{\times d}$ corresponding to $S_{deg}$ are isomorphic to a union of

\begin{center}
$\left\{ \begin{array}{cl}
\text{three copies of~} (\mathbb{P}^1)^{\times \frac{d-1}{2}} 
\text{and three copies of~} (\mathbb{P}^1)^{\times \frac{d+1}{2}}, & \text{for~} d ~\text{odd}; \\
\text{six copies of~} (\mathbb{P}^1)^{\times \frac{d}{2}},  & \text{for~} d ~\text{even}.
\end{array}\right.$
\end{center}
\end{thm}

\noindent The precise description of $V_d$ as a subset of $(\mathbb{P}^2)^{\times d}$ is provided in Proposition \ref{prop:gammadparam}. Furthermore, this scheme is not a disjoint union and Remark \ref{rmk:Sing(V_d)} describes the singularity locus of $V_d$.

In the language of \cite{RZ}, observe that the point scheme data of degenerate Sklyanin algebras does not stabilize to produce a projective scheme (of finite type) and as a consequence we cannot construct a tcr associated to $S_{deg}$. Instead, we use the truncated point schemes $V_d$ produced in Theorem \ref{thm:truncptsch} and a method from \cite[page 19]{ATV1} to form the $\N$-graded, associative ring $B$ defined below.

\begin{defn} \label{def:ptparam} The {\it point parameter ring} $B = \bigoplus_{d \geq 0} B_d$ is a ring associated to the sequence of subschemes $V_d$ of $(\mathbb{P}^2)^{\times d}$ (Definition \ref{def:trunc}). We have $B_d = H^0(V_d, \mc{L}_d)$ where $\mc{L}_d$ is the restriction of invertible sheaf $$pr_1^{\ast}\mc{O}_{\mathbb{P}^2}(1) \ten \dots \ten pr_d^{\ast}\mc{O}_{\mathbb{P}^2}(1) \cong \mc{O}_{(\mathbb{P}^2)^{\times d}}(1,\dots,1)$$ to $V_d$. The multiplication map $B_i \times B_j \ra B_{i+j}$ is defined by applying $H^0$ to the isomorphism $pr_{1,\dots,i}(\mc{L}_i) \ten_{\mc{O}_{V_{i+j}}} pr_{i+1,\dots,i+j}(\mc{L}_j) \ra \mc{L}_{i+j}$.
\end{defn}

Despite point parameter rings not being well understood in general, the final section of this paper verifies the following properties of $B=B(S_{deg})$.

\begin{thm} \label{thm:BforSdeg} The point parameter ring $B$ for a degenerate three-dimen-sional Sklyanin algebra $S_{deg}$ has Hilbert series $H_B(t)=\frac{(1+t^2)(1+2t)}{(1-2t^2)(1-t)}$ and is generated in degree one.
\end{thm}

Hence we have a surjection of $S_{deg}$ onto $B$, which is akin to the result involving Sklyanin algebras and corresponding tcrs (Theorem \ref{thm:Skl3}).

\begin{cor} \label{cor:BforSdeg} The ring $B= B(S_{deg})$ has exponential growth and therefore infinite GK dimension. Moreover $B$ is neither right Noetherian, Koszul, nor a domain. 
Furthermore $B$ is a factor of the corresponding $S_{deg}$ by an ideal $K$ where $K$ has six generators of degree 4 (and possibly more of higher degree).
\end{cor}

Therefore the behavior of $B(S_{deg})$ resembles that of $S_{deg}$. It is natural to ask if other noncommutative algebras can be analyzed in a similar fashion, though we will not address this here.
\medskip

\noindent {\bf Acknowledgements.} I sincerely thank my advisor Toby Stafford for introducing me to this field and for his encouraging advice on this project. I am also indebted to Karen Smith for supplying many insightful suggestions. I have benefited from conversations with Hester Graves, Brian Jurgelewicz, and Sue Sierra, and I thank them.


\section{Structure of degenerate Sklyanin algebras}

In this section, we establish Proposition \ref{prop:S(1bc)}. We begin by considering the degenerate Sklyanin algebras $S(a,b,c)_{deg}$ with $a^3=b^3=c^3=1$ (Definition \ref{def:Skl3}) and the following definitions from \cite{GW}.

\begin{defn} \label{def:Ore} Let $\alpha$ be an endomorphism of a ring $R$. An {\it $\alpha$-derivation} on $R$ is any additive map $\delta: R \ra R$ so that $\delta(rs) = \alpha(r)\delta(s) + \delta(r)s$ for all $r,s \in R$. The set of $\alpha$-derivations of $R$ is denoted $\alpha$-Der($R$).

We write $S=R[z; \alpha, \delta]$ provided $S$ is isomorphic to the polynomial ring $R[z]$ as a left $R$-module but with multiplication given by $zr= \alpha(r)z + \delta(r)$ for all $r\in R$. Such a ring $S$ is called an {\it Ore extension of $R$}.
\end{defn}

By generalizing the work of \cite{BjSta} we see that most degenerate Sklyanin algebras are factors of Ore extensions of the free algebra on two variables.
 
\begin{prop} \label{prop:S(1bc)2} In the case of $a^3=b^3=c^3=1$, assume without loss of generality $a=1$. Then for $[1:b:c] \in \mf{D}$ we get the ring isomorphism
\begin{equation} \label{eq:S(1bc)}
S(1,b,c) \cong \frac{k\{x,y\}[z; \alpha, \delta]}{(\Omega)}
\end{equation}
where $\alpha \in $End($k\{x,y\}$) is defined by $\alpha(x)= -bx$, $\alpha(y)= -b^2y$ and the element $\delta \in \alpha$-Der($k\{x,y\}$) is given by $\delta(x) = -cy^2$, $\delta(y) = -b^2cx^2$. Here  $\Omega = xy+ byx +cz^2$ is a normal element of $k\{x,y\}[z, \alpha, \delta]$.
\end{prop}

\begin{proof}
By direct computation $\alpha$ and $\delta$ are indeed an endomorphism and $\alpha$-derivation of $k\{x,y\}$ respectively. Moreover $x \cdot \Omega = \Omega \cdot bx$, ~$y \cdot \Omega = \Omega \cdot by$, ~$z \cdot \Omega = \Omega \cdot z$ so $\Omega$ is a normal element of the Ore extension. Thus both rings of (\ref{eq:S(1bc)}) have the same generators and relations.
\end{proof}

\begin{rmk} \label{rmk:S(1bc)}
Some properties of degenerate Sklyanin algebras are easy to verify without use of the Proposition \ref{prop:S(1bc)2}. Namely one can find a basis of irreducible monomials via Bergman's Diamond lemma \cite[Theorem 1.2]{Be} to imply $\dim_k S_d = 2^{d-1}3$ for $d \geq 1$. Equivalently $S(1,b,c)$ is free with a basis $\{1,z\}$ as a left or right module over $k\{x,y\}$. Therefore, $H_{S_{deg}}(t) = \frac{1+t}{1-2t}$.
\end{rmk}

Therefore due to Proposition \ref{prop:S(1bc)2} (for $a^3=b^3=c^3=1$) or Remark \ref{rmk:S(1bc)} we have the following immediate consequence.

\begin{cor} \label{cor:S(1bc)}
The degenerate Sklyanin algebras have exponential growth, infinite GK dimension, and are not right Noetherian. Furthermore $S_{deg}$ is not a domain.
\end{cor}

\begin{proof}
The growth conditions follow from Remark \ref{rmk:S(1bc)} and the non-Noetherian property holds by \cite[Theorem 0.1]{SteZ}. Moreover if $[a:b:c] \in \{[1:0:0],$\linebreak $[0:1:0],~[0:0:1]\},$ then the monomial algebra $S(a,b,c)$ is obviously not a domain. On the other hand if $[a:b:c]$ satisfies $a^3 = b^3 = c^3=1$, then assume without loss of generality that $a=1$. As a result we have
$$f_1 + b f_2 + c f_3 ~=~ (x+by+bc^2z)(cx+cy+b^2z),$$
where $f_1=yz+bzy+cx^2$, $f_2=zx+bxz+cy^2$, and $f_3=xy+byx+cz^2$ are the relations of $S(1,b,c)$.
\end{proof}

Now we verify homological properties of degenerate Sklyanin algebras.

\begin{defn} \label{def:Koszul} Let $A$ be a cg algebra which is locally finite (dim$_k A_i < \infty$). When provided a minimal resolution of the left $A$-module $A/\bigoplus_{i \geq 1} A_i \cong k$ determined by matrices $M_i$, we say $A$ is {\it Koszul} if the entries of the $M_i$ all belong to $A_1$. 
\end{defn}

\begin{prop} \label{prop:kos,gldim} The degenerate Sklyanin algebras are Koszul with infinite global dimension.
\end{prop}

\begin{proof} For $S=S(a,b,c)$ with $a^3=b^3=c^3=1$, consider the description of $S$ in Proposition \ref{prop:S(1bc)2}. Since $k\{x,y\}$ is Koszul, the Ore extension $k\{x,y\}[z, \alpha, \delta]$ is also Koszul \cite[Definition 1.1, Theorem 10.2]{CSh}. By Proposition \ref{prop:S(1bc)2}, the element $\Omega$ is normal and regular in $k\{x,y\}[z;\alpha, \delta]$. Hence the factor $S$ is Koszul by \cite[Theorem 1.2]{ShT}. 

To conclude gl.dim($S) = \infty$, note that the Koszul dual of $S$ is 

$$S(1,b,c)^{!} \cong \frac{k\{x,y,z\}}{\left(\begin{array}{ccc} 
z^2 - cxy, && yz -c^2x^2,\\ zy - b^2yz, && y^2-bcxz,\\ zx-bxz, && yx-b^2xy \end{array}\right)}.$$

\noindent Taking the ordering $x < y < z$, we see that all possible ambiguities of $S^{!}$ are resolvable in the sense of \cite{Be}. Bergman's Diamond lemma \cite[Theorem 1.2]{Be} implies  that $S^{!}$ has a basis of irreducible monomials $\{x^i, x^jy, x^kz\}_{i,j,k \in \N}$. Hence $S^{!}$ is not a finite dimensional $k$-vector space and by \cite[Corollary 5]{Kr}, $S$ has infinite global dimension.

For $S=S(a,b,c)$ with $[a:b:c] \in \{[1:0:0],[0:1:0],[0:0:1]\}$,
note that $S$ is Koszul as its ideal of relations is generated by quadratic monomials \cite[Corollary 4.3]{PP}. Denote these monomials $m_1$, $m_2$, $m_3$.
The Koszul dual of $S$ in this case is

$$S^{!} \cong \frac{k\{x,y,z\}}{\left(\text{the six monomials not equal to~}m_i\right)}.$$

\noindent Since $S^!$ is again a monomial algebra, it contains no hidden relations and has a nice basis of irreducible monomials. In particular, $S^!$ contains $\bigoplus_{i\geq 0} kw_i$ where $w_i$ is the length $i$ word:

\[
w_i =
\begin{cases}
\underbrace{xyzxyzx\dots}_{i}, &\text{if~} [a:b:c] = [1:0:0]\\
\underbrace{xzyxzyx\dots}_{i}, &\text{if~} [a:b:c] = [0:1:0]\\
x^i, &\text{if~} [a:b:c] = [0:0:1].
\end{cases}
\]

\noindent Therefore $S^{!}$ is not a finite dimensional $k$-vector space. By \cite[Corollary 5]{Kr}, the three remaining degenerate Sklyanin algebras are of infinite global dimension.
\end{proof}


\section{Truncated point schemes of $S_{deg}$}

\noindent The goal of this section is to construct the family of truncated point schemes $\{V_d \subseteq (\mathbb{P}^2)^{\times d} \}$ associated to the degenerate three-dimensional Sklyanin algebras $S_{deg}$ (see Definition \ref{def:Sdeg}). These schemes will be used in $\S4$ for the construction of a generalized twisted homogeneous coordinate ring, namely the point parameter ring (Definition \ref{def:ptparam}). Nevertheless the family $\{V_d\}$ has immediate importance for understanding point modules over $S=S_{deg}$.

\begin{defn} \label{def:truncptmod}
A graded left $S$-module $M$ is called a {\it point module} if $M$ is cyclic and $H_M(t) = \sum_{i=0}^{\infty} t^i = \frac{1}{1-t}$. Moreover a graded left $S$-module $M$ is called a {\it truncated point module of length} $d$ if $M$ is again cyclic and $H_M(t) = \sum_{i=0}^{d-1} t^i$.
\end{defn}

Note that point modules share the same Hilbert series as a point in projective space in Classical Algebraic Geometry.

Now we proceed to construct schemes $V_d$ that will parameterize length $d$ truncated point modules. This yields information regarding point modules over $S(a,b,c)$ for {\it any} $[a:b:c] \in \mathbb{P}^2$ due to the following result. 

\begin{lem} \cite[Proposition 3.9, Corollary 3.13]{ATV1} \label{lem:gamma} Let $S = S(a,b,c)$ for any $[a:b:c] \in \mathbb{P}^2$.
Denote by $\Gamma$ the set of isomorphism classes of point modules over $S$ and 
$\Gamma_d$ the set of isomorphism classes of truncated point modules of length $d+1$. With respect to the truncation function $\rho_d: \Gamma_d \ra \Gamma_{d-1}$ given by $M \mapsto M/M_{d+1}$, we have that $\Gamma$ is the projective limit of $\{\Gamma_d\}$ as a set.  
\end{lem}

The sets $\Gamma_d$ can be understood by the schemes $V_d$ defined below.

\begin{defn}\cite[\S3]{ATV1} \label{def:truncptsch} The {\it truncated point scheme of length d},
$V_d \subseteq (\mathbb{P}^2)^{\times d}$, is the scheme defined by the  multilinearizations of relations of $S(a,b,c)$ from Definition \ref{def:Skl3}. More precisely $V_d = \V(f_i, g_i, h_i)_{0 \leq i \leq d-2}$ where 
\begin{equation} \label{eq:multilin,rel}
\begin{aligned}
f_i &= ay_{i+1}z_i + b z_{i+1}y_i + cx_{i+1}x_i\\
g_i &= az_{i+1}x_i + b x_{i+1}z_i + cy_{i+1}y_i\\
h_i &= ax_{i+1}y_i + b y_{i+1}x_i + cz_{i+1}z_i.
\end{aligned}
\end{equation}
For example, $V_1 = \mathbb{V}(0) \subseteq \mathbb{P}^2$ so we have $V_1 = \mathbb{P}^2$.  Similarly, $V_2 = \mathbb{V}(f_0,g_0,h_0) \subseteq \mathbb{P}^2 \times \mathbb{P}^2$.
\end{defn}

\begin{lem} \cite{ATV1} \label{lem:gammad}
The set $\Gamma_d$ is parameterized by the scheme $V_d$.
\end{lem}

In short, to understand point modules over $S(a,b,c)$ for any $[a:b:c] \in \mathbb{P}^2$, Lemmas \ref{lem:gamma} and \ref{lem:gammad} imply that we can now restrict our attention to truncated point schemes $V_d$. 

On the other hand, we point out another useful result pertaining to $V_d$ associated to $S(a,b,c)$ for any $[a:b:c] \in \mathbb{P}^2$.

\begin{lem} \label{lem:V_dinE^d} The truncated point scheme $V_d$ lies in $d$ copies of $E \subseteq \mathbb{P}^2$
where $E$ is the cubic curve
$E: (a^3+b^3+c^3)xyz - (abc)(x^3+y^3+z^3) =0$.
 \end{lem}

\begin{proof} Let $p_i$ denote the point $[x_i : y_i : z_i] \in \mathbb{P}^2$ and
\begin{equation} \label{eq:Mabc}
\mathbb{M}_{abc, i} := \mathbb{M}_{i} := 
{\footnotesize \left(\begin{array}{ccc} 
cx_i & az_i & by_i\\ 
bz_i & cy_i & ax_i\\ 
ay_i & bx_i & cz_i
\end{array}\right)} 
\in \text{Mat}_3(kx_i \oplus ky_i \oplus kz_i).
\end{equation}

\noindent A $d$-tuple of points $p = (p_0, p_1, \dots, p_{d-1}) \in V_d \subseteq (\mathbb{P}^2)^{\times d}$ must satisfy the system $f_i = g_i = h_i = 0$ for $0\leq i \leq d-2$ by definition of $V_d$. In other words, one is given
$\mathbb{M}_{abc, j} \cdot (x_{j+1}~~y_{j+1}~~z_{j+1})^T = 0$ or equivalently $(x_{j}~~y_{j}~~z_{j}) \cdot \mathbb{M}_{abc, j+1}= 0$ 
for $0 \leq j \leq d-2$. Therefore for $0 \leq j \leq d-1$, $\det(\mathbb{M}_{abc, j}) = 0$. This implies $p_j \in E$ for each $j$. Thus $p \in E^{\times d}$.
\end{proof}


\subsection{On the truncated point schemes of some $S_{deg}$}

We will show that to study the truncated point schemes $V_d$ of degenerate Sklyanin algebras, it suffices to understand the schemes of specific four degenerate Sklyanin algebras. Recall that $V_d$ parameterizes length $d$ truncated point modules (Lemma \ref{lem:gammad}). Moreover
note that according to \cite{Zh}, two graded algebras $A$ and $B$ have equivalent graded left module categories ($A$-Gr and $B$-Gr) if $A$ is a Zhang twist of $B$. The following is a special case of \cite[Theorem 1.2]{Zh}.

\begin{thm}\label{thm:Zhang} 
Given a $\Z$-graded $k$-algebra $S= \bigoplus_{n \in \Z}S_n$ with graded automorphism $\sigma$ of degree 0 on $S$, we form a {\it Zhang twist} $S^{\sigma}$ of $S$ by preserving the same additive structure on $S$ and defining multiplication $\ast$ as follows:
$a \ast b = ab^{\sigma^n}$ for $a \in S_n$. Furthermore if $S$ and $S^{\sigma}$ are cg and generated in degree one, then $S$-Gr and $S^{\sigma}$-Gr are equivalent categories. \qed
\end{thm}

Realize $\mf{D}$ from Definition \ref{def:Skl3} as the union of three point sets $Z_i$:
\begin{equation} \label{eq:Z_i}
\begin{array}{lll}
Z_1 := \{[1:1:1], &[1:\zeta:\zeta^2], &[1:\zeta^2:\zeta]\},\\
Z_2 := \{[1:1:\zeta], &[1:\zeta:1], &[1:\zeta^2:\zeta^2]\},\\
Z_3 := \{[1:\zeta:\zeta], &[1:1:\zeta^2], &[1:\zeta^2:1]\},\\
Z_0 := \{[1:0:0], &[0:1:0], &[0:0:1]\}.
\end{array}
\end{equation}
\noindent where $\zeta = e^{2 \pi i /3}$.
Pick respective representatives $[1:1:1]$, $[1:1: \zeta]$, $[1: \zeta: \zeta]$, and $[1:0:0]$ of $Z_1$, $Z_2$, $Z_3$, and $Z_0$. 

\begin{lem} \label{lem:fourSdeg} Every degenerate Sklyanin algebra is a Zhang twist of one the following algebras: $S(1,1,1)$, $S(1,1,\zeta)$, $S(1,\zeta, \zeta)$, and $S(1,0,0)$.
\end{lem}

\begin{proof} A routine computation shows that the following graded automorphisms of degenerate $S(a,b,c)$,
\begin{center}
$\sigma: \{x \mapsto \zeta x, ~y \mapsto \zeta^2 y, ~z \mapsto z\}$
~~and~~
$\tau: \{x \mapsto y, ~y \mapsto z,  ~z \mapsto x\},$
\end{center}
yield the Zhang twists:
\begin{center}
$\begin{array}{lllll}
S(1,1,1)^{\sigma} = S(1,\zeta,\zeta^2), && S(1,1,1)^{\sigma^{-1}}=S(1,\zeta^2, \zeta) && \text{for~} Z_1;\\
S(1,1,\zeta)^{\sigma} = S(1,\zeta,1), && S(1,1,\zeta)^{\sigma^{-1}}=S(1,\zeta^2, \zeta^2) && \text{for~} Z_2;\\
S(1,\zeta,\zeta)^{\sigma} = S(1,\zeta^2,1), && S(1,\zeta,\zeta)^{\sigma^{-1}}=S(1,1, \zeta^2) && \text{for~} Z_3;\\
S(1,0,0)^{\tau} = S(0,1,0), && S(1,0,0)^{\tau^{-1}}=S(0,0,1) && \text{for~} Z_0. 
\end{array}$
\end{center}
\vspace{-.25in}

\end{proof}
Therefore it suffices to study a representative of each of the four classes of degenerate three-dimensional Sklyanin algebras due to Theorem \ref{thm:Zhang}.


\subsection{Computation of $V_d$ for $S(1,1,1)$}

We now compute the truncated point schemes of $S(1,1,1)$ in detail. Calculations for the other three representative degenerate Sklyanin algebras, $S(1,1,\zeta)$, $S(1,\zeta, \zeta)$, $S(1,0,0)$, will follow with similar reasoning. To begin we first discuss how to build a truncated point module $M'$ of length $d$, when provided with a truncated point module $M$ of length $d-1$.

Let us explore the correspondence between truncated point modules and truncated point schemes for a given $d$; say $d \geq 3$. When given a truncated point module $M = \bigoplus_{i=0}^{d-1} M_i \in \Gamma_{d-1}$, multiplication from $S=S(a,b,c)$ is determined by a point $p=(p_0, \dots, p_{d-2}) \in V_{d-1}$ (Definition \ref{def:truncptsch},  (\ref{eq:Mabc})) in the following manner. As $M$ is cyclic, $M_i$ has basis say $\{m_i\}$. Furthermore for $x,y,z \in S$ with $p_i = [x_i:y_i:z_i] \in \mathbb{P}^2$, we get the left $S$-action on $m_i$ determined by $p_i$:
\begin{equation} \label{eq:Saction}
\begin{aligned}
x\cdot m_i = x_i m_{i+1}, && x\cdot m_{d-1} = 0;\\
y\cdot m_i = y_i m_{i+1}, && y\cdot m_{d-1} = 0;\\
z\cdot m_i = z_i m_{i+1}, && z\cdot m_{d-1} = 0.
\end{aligned}
\end{equation}
Conversely given a point $p = (p_0, \dots, p_{d-2}) \in V_{d-1}$, one can build a module $M \in \Gamma_{d-1}$ unique up to isomorphism by reversing the above process. We summarize this discussion in the following remark.

\begin{rmk}\label{con_t.pt.mod}
Refer to notation from Lemma \ref{lem:gamma}. To construct $M' \in \Gamma_{d}$ from $M \in \Gamma_{d-1}$ associated to $p \in V_{d-1}$, we require $p_{d-1} \in \mathbb{P}^2$ such that $p'=(p,p_{d-1}) \in V_d$. 
\end{rmk}

Now we begin to study the behavior of truncated point modules over $S_{deg}$ through the examination of truncated point schemes in the next two lemmas.

\begin{lem} \label{lem:notinZ_i} Let $p=(p_0, \dots, p_{d-2}) \in V_{d-1}$ with $p_{d-2} \not \in Z_i$ (refer to (\ref{eq:Z_i})). Then there exists a unique $p_{d-1} \in Z_i$ so that $p' := (p, p_{d-1}) \in V_d$. \end{lem}

\begin{proof}[Proof of \ref{lem:notinZ_i}] For $Z_1$, we study the representative algebra $S(1,1,1)$. If such a $p_{d-1}$ exists, then $f_{d-2}=g_{d-2}=h_{d-2}=0$ so we would have $$\mathbb{M}_{111, d-2} \cdot (x_{d-1}~~y_{d-1}~~z_{d-1})^T = 0$$
(Definition \ref{def:truncptsch}, Eq. (\ref{eq:Mabc})).
 Since $\rank(\mathbb{M}_{111,d-2}) = 2$ when $p_{d-2} \not \in \mf{D}$, the tuple $(x_{d-1}, ~y_{d-1}, ~z_{d-1})$ is unique up to scalar multiple and thus the point $p_{d-1}$ is unique. 
 
To verify the existence of $p_{d-1}$, say $p_{d-2} = [0: y_{d-2} : z_{d-2}]$. We require $p_{d-2}$ and $p_{d-1}$ to satisfy the system of equations: 
\begin{center}
$\begin{array}{rll}
f_{d-2}=g_{d-2}=h_{d-2}=0 && \text{(Eq. (\ref{eq:multilin,rel}))}\\
y_{d-2}^3 + z_{d-2}^3 = x_{d-1}^3 + y_{d-1}^3 + z_{d-1}^3 = 0 && \text{($p_{d-2}$, $p_{d-1} \in E$, Lemma \ref{lem:V_dinE^d}).} 
\end{array}$
\end{center}
However basic algebraic operations imply $y_{d-2} = z_{d-2} = 0$, thus producing a contradiction.  
Therefore, without loss of generality $p_{d-2} = [1 : y_{d-2} : z_{d-2}]$. With similar reasoning we must examine the system
\begin{equation} \label{eq:system}
\begin{array}{rl}
y_{d-1} z_{d-2} + z_{d-1}y_{d-2} + x_{d-1} &= 0\\
z_{d-1} + x_{d-1} z_{d-2} + y_{d-1} y_{d-2} &= 0\\
x_{d-1} y_{d-2} +y_{d-1} + z_{d-1} z_{d-2} &= 0\\
1 + y_{d-2}^3 + z_{d-2}^3 &= 3y_{d-2}z_{d-2}\\
x_{d-1}^3 + y_{d-1}^3 + z_{d-1}^3 &= 3x_{d-1}y_{d-1}z_{d-1}.
\end{array}
\end{equation}

\noindent There are three solutions $(p_{d-2}, p_{d-1}) \in (E \setminus Z_1) \times E$ to (\ref{eq:system}): 
\[
\left\{
\begin{array}{rl}
([1:-(1+z_{d-2}):z_{d-2}], &[1:1:1]),\\
([1: -\zeta(1 + \zeta z_{d-2}) : z_{d-2}], &[1:\zeta:\zeta^2]),\\
([1: -\zeta(\zeta+z_{d-2}) : z_{d-2}], &[1:\zeta^2:\zeta])
\end{array}
\right\}.
\]
Thus when $p_{d-2} \not \in Z_1$, there exists an unique point $p_{d-1} \in Z_1$ so that $(p_0, \dots, p_{d-2}, p_{d-1}) \in V_d$.

Now having studied $S(1,1,1)$ with care, we leave it to the reader to verify the assertion for the algebras $S(1,1, \zeta)$, $S(1,\zeta,\zeta)$, and $S(1,0,0)$ in a similar manner.
\end{proof}

The next result explores the case when $p_{d-2} \in Z_i$.

\begin{lem} \label{lem:inZ_i} Let $p=(p_0, \dots, p_{d-2}) \in V_{d-1}$ with $p_{d-2} \in Z_i$. Then for any $[y_{d-1} : z_{d-1}] \in \mathbb{P}^1$ there exists a function $\theta$ of two variables so that $$p_{d-1} = [\theta(y_{d-1},z_{d-1}) : y_{d-1} : z_{d-1}] \not \in Z_i$$ which satisfies $(p_0,\dots, p_{d-2}, p_{d-1}) \in V_d$.
\end{lem}

\begin{proof} The point $p'= (p,p_{d-1}) \in V_d$ needs to satisfy $f_i = g_i = h_i =0$ for $0 \leq i \leq d-2$ (Definition \ref{def:truncptsch}). Since $p\in V_{d-1}$, we need only to consider the equations $f_{d-2} = g_{d-2} = h_{d-2} = 0$ with $p_{d-2} \in Z_i$. 
 
We study $S(1,1,1)$ for $Z_1$ so the relevant system of equations is 
\begin{center}
$\begin{array}{l}
f_{d-2}:~ y_{d-1}z_{d-2} + z_{d-1}y_{d-2} + x_{d-1}x_{d-2} = 0\\
g_{d-2}:~ z_{d-1}x_{d-2} + x_{d-1}z_{d-2} + y_{d-1}y_{d-2} = 0\\
h_{d-2}:~ x_{d-1}y_{d-2} + y_{d-1}x_{d-2} + z_{d-1}z_{d-2} = 0.
\end{array}$
\end{center}

\noindent If $p_{d-2} = [1:1:1] \in Z_1$, then $x_{d-1} = -(y_{d-1} + y_{d-1})$ is required. On the other hand, if $p_{d-2} = [1:\zeta:\zeta^2]$ or $[1:\zeta^2:\zeta]$, we require $x_{d-1} = -\zeta(y_{d-1} + \zeta z_{d-1})$ or $x_{d-1} = -\zeta( \zeta y_{d-1} + z_{d-1})$ respectively. Thus our function $\theta$ is defined as 
\begin{center}
$\theta(y_{d-1}, z_{d-1}) = \left\{\begin{array}{ll}
                      -(y_{d-1} + z_{d-1}), & \text{if~} p_{d-2} = [1:1:1]\\
                      -(\zeta y_{d-1} + \zeta^2 z_{d-1}), & \text{if~} p_{d-2} = [1: \zeta                             :\zeta^2]\\
                      -(\zeta^2 y_{d-1} + \zeta z_{d-1}), & \text{if~} p_{d-2} = [1 : \zeta^2                          : \zeta].
                      \end{array} \right.$\\
                      \end{center}
                      
\noindent The arguments for $S(1,1,\zeta)$, $S(1, \zeta, \zeta)$, and $S(1,0,0)$ proceed in a likewise fashion.                      
\end{proof}

Fix a pair ($S_{deg}$, $Z_i(S_{deg})$). We now know if $p_{d-2} \not \in Z_i$, then from every truncated point module of length $d$ over $S_{deg}$ we can produce a unique truncated point module of length $d+1$. Otherwise if $p_{d-2} \in Z_i$, we get a $\mathbb{P}^1$ worth of length $d+1$ modules. We summarize this in the following statement which is made precise in Proposition \ref{prop:gammadparam}.

\begin{prop} \label{prop:paramspace}
The parameter space of $\Gamma_d$ over $S_{deg}$ is isomorphic to the singular and nondisjoint
union of
\begin{center}
$\left\{ \begin{array}{cl}
\text{three copies of~} (\mathbb{P}^1)^{\times \frac{d-1}{2}} 
\text{and three copies of~} (\mathbb{P}^1)^{\times \frac{d+1}{2}}, & \text{for~} d ~\text{odd}; \\
\text{six copies of~} (\mathbb{P}^1)^{\times \frac{d}{2}},  & \text{for~} d ~\text{even}.
\end{array}\right.$
\end{center}
\end{prop}

\noindent The detailed statement and proof of this proposition will follow from the results below. We restrict our attention to $S(1,1,1)$ for reasoning mentioned in the proofs of Lemmas \ref{lem:notinZ_i} and \ref{lem:inZ_i}.

\subsubsection{Parameterization of $\Gamma_2$}
Recall that length 3 truncated point modules of $\Gamma_2$ are in bijective correspondence to points on $V_2 \subset \mathbb{P}^2 \times \mathbb{P}^2$ (Lemma \ref{lem:gammad}) and it is our goal to depict this truncated point scheme.  
By Lemma \ref{lem:V_dinE^d}, we know that $V_2 \subseteq E \times E$. Furthermore note that with $\zeta = e^{2 \pi i /3}$, the curve $E = E_{111}$ is the union of three projective lines: 
{\small \begin{equation} \label{eq:P^1s}
\mathbb{P}^1_A: x = -(y+z), ~~~~~~\mathbb{P}^1_B: x= -(\zeta y+\zeta^2 z), ~~~~~\mathbb{P}^1_C: x = -(\zeta^2 y + \zeta z)
\end{equation}}
\vspace{-.3in}

\[
\xymatrix@-2pc{
 && & {}^{\mathbb{P}^1_C} \ar@{<->}[rrrddddddddddddd]& & {}^{\mathbb{P}^1_B} \ar@{<->}[lllddddddddddddd] & & &  \\
 & & &  & & & & &\\
  & & &  & & & & &\\
 & & & & {}^{\text{[1:1:1]}}& & & & \\
  & & &  & & & & &\\
    & & &  & & & & &\\
      & & &  & & & &\\
        & & &  & & & & \\
 & & & \overset{}{{}_{\text{[1:$\zeta$:$\zeta^2$]}}}& &  \overset{}{{}_{\text{[1:$\zeta^2$:$\zeta$]}}}&&&\\ 
\ar@{<->}[rrrrrrrr] & &  & & & &  & & {}^{\mathbb{P}^1_A} \\
  & & &  & & & & &\\
   & & &  & & & & &\\
 & & &  & & & & &\\
 & & & & & &                 
}
\]
\vspace{-.4in}

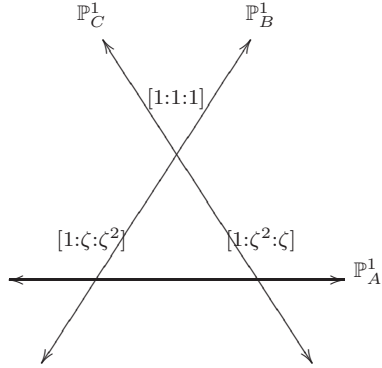
$$\text{Figure 1: The curve~} E=E_{111}  \subseteq \mathbb{P}^2:~ x^3+y^3+z^3-3xyz=0.$$
\vspace{-.1in}

Now to calculate $V_2$, recall that $\Gamma_2$ consists of length 3 truncated point modules $M_{(3)} := M_0 \oplus M_1 \oplus M_2$ where $M_i$ is a 1-dimensional $k$-vector space say with basis $m_i$. The module $M_{(3)}$ has action determined by $(p_0, p_1) \in V_2$ (Eq. (\ref{eq:Saction})).
Moreover Lemmas \ref{lem:notinZ_i} and \ref{lem:inZ_i} provide the precise conditions for $(p_0,p_1)$ to lie in $E \times E$. Namely, 

\begin{lem} \label{lem:gamma2param} Refer to (\ref{eq:P^1s}) for notation. The set of length 3 truncated point modules $\Gamma_2$ is parametrized by the scheme $V_2 = \mathbb{V}(f_0,g_0,h_0)$ which is the union of the six subsets:
\begin{center}
$\begin{array}{lll}
\mathbb{P}^1_A \times [1:1:1]; 
&&[1:1:1]\times \mathbb{P}^1_A;\\
\mathbb{P}^1_B \times [1:\zeta: \zeta^2]; 
&&[1: \zeta :\zeta^2] \times \mathbb{P}^1_B;\\
\mathbb{P}^1_C \times [1:\zeta^2: \zeta];
&&[1: \zeta^2 :\zeta] \times \mathbb{P}^1_C.
\end{array}$
\end{center}
\noindent  of $E \times E$. Thus $\Gamma_2$ is isomorphic to 6 copies of $\mathbb{P}^1$. \qed
\end{lem} 

\subsubsection{Parameterization of $\Gamma_d$ for general $d$}

To illustrate the parametrization of $\Gamma_d$, we begin with a truncated point module $M_{(d+1)}$ of length $d+1$ corresponding to $(p_0, p_1,\dots,p_{d-1}) \in V_d \subseteq (\mathbb{P}^2)^{\times d}$. Due to Lemmas \ref{lem:V_dinE^d}, \ref{lem:notinZ_i}, and \ref{lem:inZ_i}, we know that $(p_0, p_1,\dots,p_{d-1})$ belongs to either {\small $$\underbrace{(E \setminus Z_1) \times Z_1 \times (E \setminus Z_1) \times Z_1 \times \dots}_d \text{~~or~} \underbrace{Z_1 \times (E \setminus Z_1) \times Z_1 \times (E \setminus Z_1) \times \dots}_d$$} 
\hspace{-.1in} where $Z_1$ is defined in  (\ref{eq:Z_i}). 

By adapting the notation of Lemma \ref{lem:inZ_i}, we get in the first case that the point $(p_0, p_1, \dots, p_{d-1})$ is of the form 
{\small $$([\theta(y_0,z_0) : y_0: z_0], ~[1:\omega:\omega^2], ~[\theta(y_2,z_2) : y_2 :z_2], ~[1:\omega:\omega^2], \dots) \in (\mathbb{P}^2)^{\times d}$$}
\hspace{-.1in} where $\omega^3=1$ and $\theta(y,z) = -(\omega y + \omega^2 z)$. Thus in this case, the set of length $d$ truncated point modules is parameterized by three copies of $(\mathbb{P}^1)^{\times \lceil d/2 \rceil}$ with coordinates $([y_0:z_0],[y_2:z_2], \dots, [y_{2\lceil d/2 \rceil -1} : z_{2\lceil d/2 \rceil -1}])$.  

In the second case $(p_0, p_1, \dots, p_{d-1})$ takes the form
{\small $$([1:\omega:\omega^2], ~[\theta(y_1,z_1): y_1 :z_1], ~[1:\omega:\omega^2], [\theta(y_3,z_3): y_3 :z_3], \dots) \in (\mathbb{P}^2)^{\times d}$$} 
\hspace{-.12in} and the set of truncated point modules is parameterized with three copies of $(\mathbb{P}^1)^{\times \lfloor d/2 \rfloor}$ with coordinates $([y_1:z_1], [y_3:z_3], \dots, [y_{2\lfloor d/2 \rfloor -1} : z_{2\lfloor d/2 \rfloor -1}])$. 

In other words, we have now proved the next result.

\begin{prop} \label{prop:gammadparam} Refer to  (\ref{eq:P^1s}) for notation. For $d \geq 2$ the truncated point scheme $V_d$ for $S(1,1,1)$ is equal to the union of the six subsets $\bigcup_{i=1}^6 W_{d,i}$ of $(\mathbb{P}^2)^{\times d}$ where
\begin{center}
$\begin{array}{ll}
W_{d,1} &= \mathbb{P}^1_A \times [1:1:1] \times \mathbb{P}^1_A \times [1:1:1] \times \dots,\\

W_{d,2} &= [1:1:1]\times \mathbb{P}^1_A \times [1:1:1] \times \mathbb{P}^1_A \times \dots,\\

W_{d,3} &= \mathbb{P}^1_B \times [1:\zeta: \zeta^2] \times \mathbb{P}^1_B \times [1:\zeta:\zeta^2] \times \dots,\\ 

W_{d,4} &= [1: \zeta :\zeta^2] \times \mathbb{P}^1_B \times [1: \zeta: \zeta^2] \times \mathbb{P}^1_C \times \dots,\\

W_{d,5} &= \mathbb{P}^1_C \times [1:\zeta^2: \zeta] \times \mathbb{P}^1_C \times [1:\zeta^2:\zeta] \times \dots,\\
 
W_{d,6} &= [1: \zeta^2 :\zeta] \times \mathbb{P}^1_C \times [1:\zeta^2:\zeta] \times \mathbb{P}^1_C \times \dots.
\end{array}$ 
\end{center} 
\vspace{-.15in}
\qed
\end{prop}
\medskip

\noindent As a consequence, we obtain the proof of Proposition \ref{prop:paramspace} for $S(1,1,1)$ and this assertion holds for the remaining degenerate Sklyanin algebras due to Lemma \ref{lem:fourSdeg}, and analogous proofs for Lemmas \ref{lem:notinZ_i} and \ref{lem:inZ_i}.\qed
\medskip

We thank Karen Smith for suggesting the following elegant way of interpreting the point scheme of $S(1,1,1)$.

\begin{rmk} \label{rmk:ptscheme}
We can provide an alternate geometric description of the point scheme of the $\Gamma$ of $S(1,1,1)$. Let $G := \Z_3 \rtimes \Z_2 = < \hspace{-.05in}\zeta, \sigma \hspace{-.05in}>$ where $\zeta  = e^{2\pi i/3}$ and $\sigma^2 =1$. We define a $G$-action on $\mathbb{P}^2 \times \mathbb{P}^2$ as follows:
\[
\begin{aligned}
\zeta([x:y:z],[u:v:w]) &= ([x:\zeta^2 y:\zeta z],[u:\zeta v:\zeta^2 w])\\
\sigma([x:y:z],[u:v:w]) &= ([u:v:w],[x:y:z])
\end{aligned}
\]
Note that $G$ stabilizes $E \times E$ and acts transitively on the $W_{2,i}$. 
We extend the action of $G$ to $(\mathbb{P}^2 \times \mathbb{P}^2)^{\times \infty}$ diagonally. Now we interpret $\Gamma$ as
$$\Gamma ~=~ \underset{\longleftarrow}{\lim} V_d ~=~ \underset{\longleftarrow}{\lim} V_{2d} ~=~ \underset{\longleftarrow}{\lim} \bigcup_i W_{2d,i} ~=~ G \cdot (\mathbb{P}_A^1 \times [1:1:1])^{\times \infty},$$
\vspace{-.2in}

\noindent as sets.

\end{rmk}


\section{Point parameter ring of $S(1,1,1)$}

We now construct a graded associative algebra $B$ from truncated point schemes of the degenerate Sklyanin algebra $S=S(1,1,1)$. The analogous result for the other degenerate Sklyanin algebras will follow in a similar fashion and we leave the details to the reader. As is true for the Sklyanin algebras themselves, it will be shown that this algebra $B$ is a proper factor of $S(1,1,1)$ and its properties  closely reflect those of $S(1,1,1)$. We will for example show that $B$ is not right Noetherian, nor a domain.

The definition of the algebra $B$ initially appears in \cite[$\S$3]{ATV1}. Recall that we have projection maps $pr_{1,\dots,d-1}$ and $pr_{2,\dots,d}$ from $(\mathbb{P}^2)^{\times d}$ to $(\mathbb{P}^2)^{\times d-1}$. Restrictions of these maps to the truncated point schemes $V_d \subseteq (\mathbb{P}^2)^{\times d}$ (Definition \ref{def:truncptsch}) yield
$$pr_{1,\dots, d-1}(V_d) \subset V_{d-1} \text{~~and~~} pr_{2,\dots, d}(V_d) \subset V_{d-1} \text{~~for all~}d.$$

\begin{defn} \label{def:ptparamS} Given the above data, the {\it point parameter ring} $B=B(S)$ is an associative $\N$-graded ring defined as follows. First $B_d = H^0(V_d, \mc{L}_d)$ where $\mc{L}_d$ is the restriction of invertible sheaf $$pr_1^{\ast}\mc{O}_{\mathbb{P}^2}(1) \ten \dots \ten pr_d^{\ast}\mc{O}_{\mathbb{P}^2}(1) \cong \mc{O}_{(\mathbb{P}^2)^{\times d}}(1,\dots,1)$$ to $V_d$. The multiplication map $\mu_{i,j}: B_i \times B_j \ra B_{i+j}$ is then defined by applying $H^0$ to the isomorphism $$pr_{1,\dots,i}^{\ast}(\mc{L}_i) \ten_{\mc{O}_{V_{i+j}}} pr_{i+1,\dots,i+j}^{\ast}(\mc{L}_j) \ra \mc{L}_{i+j}.$$ We declare $B_0 = k$.
\end{defn}

\noindent We will later see in Theorem \ref{thm:Bgenindeg1} that $B$ is generated in degree one; thus $S$ surjects onto $B$.
        
To begin the analysis of $B$ for $S(1,1,1)$, recall that $V_1 = \mathbb{P}^2$ so $$B_1 = H^0(V_1,
pr_1^{\ast}\mc{O}_{\mathbb{P}^2}(1)) = kx \oplus ky \oplus kz$$ where $[x:y:z]$
are the coordinates of $\mathbb{P}^2$. For $d \geq 2$ we will compute $\dim_k B_d$ and then proceed to the more difficult task of identifying the multiplication maps $\mu_{i,j}: B_i \times B_j \ra B_{i+j}$.
Before we get to specific calculations for $d \geq 2$, let us recall that the schemes $V_d$ are realized as the union  of six subsets $\{W_{d,i}\}_{i=1}^6$ of $(\mathbb{P}^2)^{\times d}$ described in Proposition \ref{prop:gammadparam} and Eq. (\ref{eq:P^1s}). These subsets intersect nontrivially so that each $V_d$ for $d \geq 2$ is singular. More precisely,

\begin{rmk} \label{rmk:Sing(V_d)} A routine computation shows that the singular subset, Sing($V_d$), consists of six points:
\[\begin{array}{lll}
v_{d,1}:= &([1:1:1], ~~[1:\zeta:\zeta^2], ~~[1:1:1], ~~[1:\zeta:\zeta^2], \dots )  &\in W_{d,2} \cap W_{d,3},\\

v_{d,2}:= &([1:1:1], ~~[1:\zeta^2:\zeta], ~~[1:1:1], ~~[1:\zeta^2:\zeta], \dots )  &\in W_{d,2} \cap W_{d,5},\\

v_{d,3}:= &([1:\zeta:\zeta^2], ~~[1:1:1], ~~[1:\zeta:\zeta^2], ~~[1:1:1], \dots )  &\in W_{d,1} \cap W_{d,4},\\

v_{d,4}:= &([1:\zeta:\zeta^2], ~~[1:\zeta:\zeta^2], ~~[1:\zeta:\zeta^2], ~~[1:\zeta:\zeta^2], \dots )  &\in W_{d,3} \cap W_{d,4},\\

v_{d,5}:= &([1:\zeta^2:\zeta], ~~[1:1:1], ~~[1:\zeta^2:\zeta], ~~[1:1:1], \dots ) &\in W_{d,1} \cap W_{d,6},\\

v_{d,6}:= &([1:\zeta^2:\zeta], ~~[1:\zeta^2:\zeta], ~~[1:\zeta^2:\zeta], ~~[1:\zeta^2:\zeta], \dots )  &\in W_{d,5} \cap W_{d,6}.
\end{array}\]
\noindent where $\zeta = e^{2 \pi i/3}$.
\end{rmk}


\subsection{Computing the dimension of $B_d$}

Our objective in this section is to prove

\begin{prop} \label{prop:|B_d|}
For $d \geq 1$, $\dim_k B_d
= 3\left(2^{\lfloor \frac{d+1}{2} \rfloor} +2^{\lceil \frac{d-1}{2} \rceil}\right) -6$.
\end{prop}

For the rest of the section, let $\ones$ denote a sequence of 1s of appropriate length.
Now consider the normalization morphism $\pi: V_d' \ra V_d$ where
$V_d'$ is the disjoint union of the six subsets $\{W_{d,i}\}_{i=1}^6$ mentioned in Proposition \ref{prop:gammadparam}. This map induces the following short exact sequence of sheaves on $V_d$:
\begin{equation} \label{eq:normseq}
0 \ra \mc{O}_{V_d}(\ones) \ra (\pi_{\ast} \mc{O}_{V_d'})(\ones) \ra \mc{S}(\ones) \ra 0,
\end{equation}

\noindent where $\mc{S}$ is the skyscraper sheaf whose support is Sing($V_d$), that is $\mc{S} = \bigoplus_{k=1}^6 \mc{O}_{\{v_{d,k}\}}$.
\smallskip

Note that we have
\begin{equation}\label{kvs}
H^0(V_d, (\pi_{\ast} \mc{O}_{V_d'})(\ones)) \underset{k-\text{v.s.}}{\cong} H^0(V_d', \mc{O}_{V_d'}(\ones))
\end{equation}
since the normalization morphism is a finite map, which in turn is an affine map \cite[Exercises II.5.17(b), III.4.1]{Ha}. To complete the proof of the proposition, we make the following assertion:

\smallskip

\noindent {\bf Claim}: $H^1(V_d, \mc{O}_{V_d}(\ones)) = 0$.
\smallskip

\noindent Assuming that the claim holds, we get from (\ref{eq:normseq}) the following long exact sequence of cohomology: 
\[
\begin{array}{lll}
0 \ra H^0(V_d,\mc{O}_{V_d}(\ones)) &\ra H^0(V_d,(\pi_{\ast}\mc{O}_{V'_d})(\ones))\\ &\ra H^0(V_d,\mc{S}(\ones))
 ~\ra~ H^1(V_d,\mc{O}_{V_d}(\ones))=0.
\end{array}
\]

\noindent Thus, with writing $h^0(X,\mc{L}) = \dim_k H^0(X,\mc{L})$,  (\ref{kvs}) implies that
\begin{center}
$\begin{array}{ll}
\dim_k B_d = h^0(\mc{O}_{V_d}(\ones))
&= h^0((\pi_{\ast} \mc{O}_{V_d'})(\ones)) - h^0(\mc{S}(\ones))\\
&= h^0(\mc{O}_{V_d'}(\ones)) - h^0(\mc{S}(\ones))\\
&= \sum_{i=1}^6 h^0(\mc{O}_{W_{d,i}}(\ones)) - 6.
\end{array}$
\end{center}
\medskip

\noindent Therefore applying Proposition \ref{prop:paramspace} and K\"unneth's Formula \cite[A.10.37]{BoG} completes the proof of Proposition \ref{prop:|B_d|}.
It now remains to verify the claim. 
\medskip

\noindent{\it Proof of Claim}: By the discussion above, it suffices to show that $$\delta_d: H^0(V_d', \mc{O}_{V_d'}(\ones)) \ra H^0\left(\bigcup_{k=1}^6 \{v_{d,k}\}, ~\mc{S}(\ones)\right)$$ is surjective. Referring to the notation of Proposition \ref{prop:gammadparam} and Remark \ref{rmk:Sing(V_d)}, we choose $v_{d,i} \in$ Supp$(\mc{S}(\ones))$ and $W_{d,k_i}$ containing $v_{d,i}$. This $W_{d,k_i}$ contains precisely two points of Supp($\mc{S}(\ones)$) and say the other is $v_{d,j}$ for $j\neq i$.
After choosing a basis $\{t_i\}_{i=1}^6$ for the six-dimensional vector space $H^0(\mc{S}(\ones))$ where $t_i(v_{d,j}) = \delta_{ij}$, we construct a preimage of each $t_i$. Since $\mc{O}_{W_{d,k_i}}(\ones)$ is a very ample sheaf, it separates points. In other words there exists $\tilde s_i \in H^0(\mc{O}_{W_{d,k_i}}(\ones))$ such that $\tilde s_i(v_{d,j}) = \delta_{ij}$. Extend this section $\tilde s_i$ to $s_i \in H^0(\mc{O}_{V_d'}(\ones))$ by declaring $s_i = \tilde s_i$ on $W_{d,k_i}$ and $s_i = 0$ elsewhere. Thus $\delta_d(s_i) = t_i$ for all $i$ and the map $\delta_d$ is surjective as desired. \qed

\medskip
This concludes the proof of Proposition \ref{prop:|B_d|}.

\begin{cor} \label{cor:|B_d|} We have $\lim_{d \ra \infty} (\dim_k B_d)^{1/d} = \sqrt{2} > 1$ so $B$ has exponential growth hence infinite GK dimension. By \cite[Theorem 0.1]{SteZ}, $B$ is not left or right Noetherian. \qed
\end{cor}                  

On the other hand, we can also determine the Hilbert series of $B$.

\begin{prop} \label{prop:H_B(t)} $\ds H_B(t) = \frac{(1+t^2)(1+2t)}{(1-2t^2)(1-t)}.$
\end{prop}

\begin{proof}
Recall from Proposition \ref{prop:|B_d|} that $\dim_k B_d
= 3\left(2^{\lceil \frac{d-1}{2} \rceil} +2^{\lfloor \frac{d+1}{2} \rfloor}\right) -6$ for $d \geq 1$ and that $\dim_k B_0 =1$.
Thus
{\footnotesize \begin{center}
$\begin{array}{ll}
H_B(t) &= \ds 1 + 3 \left( \sum_{d \geq 1} 2^{\lceil \frac{d-1}{2} \rceil}t^d + \sum_{d \geq 1} 2^{\lfloor \frac{d+1}{2} \rfloor} t^d - 2 \sum_{d \geq 1} t^d \right)\\

&= \ds 1 + 3 \left( t\sum_{d \geq 0} 2^{\lceil \frac{d}{2} \rceil}t^d + 2t\sum_{d \geq 0} 2^{\lfloor \frac{d}{2} \rfloor} t^d - 2t \sum_{d \geq 0} t^d \right).
\end{array}$
\end{center}}
\noindent Consider generating functions $a(t) = \sum_{d \geq 0} a_d t^d$ and $b(t)=\sum_{d \geq 0} b_d t^d$ for the respective sequences $a_d=2^{\lceil d/2 \rceil}$ and $b_d=2^{\lfloor d/2 \rfloor}$.
Elementary operations result in $a(t) = \frac{1+2t}{1-2t^2}$ and $b(t)=\frac{1+t}{1-2t^2}$. Hence
{\footnotesize
$$H_B(t) = \ds 1 + 3 \left[ t\left(\frac{1+2t}{1-2t^2} \right) + 2t\left( \frac{1+t}{1-2t^2}\right) - 2t \left(\frac{1}{1-t}\right) \right] = \ds \frac{(1+t^2)(1+2t)}{(1-2t^2)(1-t)}.$$}
\end{proof}       


\subsection{The multiplication maps $\mu_{ij}: B_i \times B_j \ra B_{i+j}$}

In this section we examine the multiplication of the point parameter ring $B$ of $S(1,1,1)$. In particular, we show that the multiplication maps are surjective which results in the following theorem.

\begin{thm} \label{thm:Bgenindeg1}
The point parameter ring $B$ of $S(1,1,1)$ is generated in degree one.
\end{thm}

\noindent With similar reasoning, $B = B(S_{deg})$ is generated in degree one for all $S_{deg}$.

\begin{proof} It suffices to prove that the multiplication maps $\mu_{d,1}: B_d \times B_1 \ra B_{d+1}$ are surjective for $d \geq 1$. Recall from Definition \ref{def:ptparamS} that $\mu_{d,1} = H^0(m_d)$ where $m_d$ is the isomorphism 
$$m_d: \mc{O}_{V_d \times \mathbb{P}^2}(1,\dots,1,0) \ten_{\mc{O}_{V_{d+1}}} \mc{O}_{(\mathbb{P}^2)^{\times d}}(0,\dots,0,1) \ra \mc{O}_{V_{d+1}}(1,\dots,1).$$
To use the isomorphism $m_d$, we employ the following commutative diagram:
\bigskip

\begin{equation}\label{eq:m_d}
\end{equation}
\vspace{-.9in}

\xymatrix{
\mc{O}_{V_d \times \mathbb{P}^2}(1,\dots,1,0) \ten_{\mc{O}_{(\mathbb{P}^2)^{\times d+1}}} \mc{O}_{(\mathbb{P}^2)^{\times d+1}}(0,\dots,0,1) \ar@{->}[d] \ar@{->}[dr]^{t_d}\\
\mc{O}_{V_d \times \mathbb{P}^2}(1,\dots,1,0) \ten_{\mc{O}_{V_{d+1}}} \mc{O}_{(\mathbb{P}^2)^{\times d+1}}(0,\dots,0,1) \ar@{->}[r]^{\hspace{1in}m_d} & \mc{O}_{V_{d+1}}(1,\dots,1).
}
\bigskip

\noindent The source of $t_d$ is isomorphic to $\mc{O}_{V_d \times \mathbb{P}^2}(1,\dots,1)$ and the map $t_d$ is given by restriction to $V_{d+1}$. Hence we have the short exact sequence
\begin{equation}\label{eq:sesfort_d}
0 \lra \mc{I}_{\frac{V_{d+1}}{V_d \times \mathbb{P}^2}}(\ones) \lra \mc{O}_{V_d \times \mathbb{P}^2}(\ones) \overset{t_d}{\lra} \mc{O}_{V_{d+1}}(\ones) \lra 0,
\end{equation}
where $\mc{I}_{\frac{V_{d+1}}{V_d \times \mathbb{P}^2}}$ is the ideal sheaf of $V_{d+1}$ defined in $V_d \times \mathbb{P}^2$.
Since the K\"unneth formula and the claim from \S4.1 implies that $H^1(\mc{O}_{V_d \times \mathbb{P}^2}(\ones)) = 0$, the cokernel of $H^0(t_d)$ is $H^1\Big(\mc{I}_{\frac{V_{d+1}}{V_d \times \mathbb{P}^2}}(\ones)\Big)$. Now we assert

\begin{prop}\label{prop:H^1(IV)}
$H^1\Big(\mc{I}_{\frac{V_{d+1}}{V_d \times \mathbb{P}^2}}(\ones)\Big) = 0$ for $d \geq 1$.
\end{prop}

\noindent By assuming that Proposition \ref{prop:H^1(IV)} holds, we get the surjectivity of $H^0(t_d)$ for $d \geq 1$. Now by applying the global section functor to Diagram (\ref{eq:m_d}), we have that $H^0(m_d) = \mu_{d,1}$ is surjective for $d \geq 1$. This concludes the proof of Theorem \ref{thm:Bgenindeg1}.
\end{proof}

\noindent{{\it Proof of \ref{prop:H^1(IV)}}.
Consider the case $d=1$. We study the ideal sheaf $\mc{I}_{\frac{V_2}{\mathbb{P}^2 \times \mathbb{P}^2}}:= \mc{I}_{V_2}$ by using the resolution of
the ideal of defining relations $(f_0,g_0,h_0)$ for $V_2$ (Eqs. (\ref{eq:multilin,rel})) in the $\N^2$-graded ring $R=k[x_0,y_0,z_0,x_1,y_1,z_1]$. Note that
each of the defining equations have bidegree (1,1) in $R$  and we get
the following resolution:
{\small \[
0 \ra \mc{O}_{\mathbb{P}^2 \times
\mathbb{P}^2}(-3,-3) \ra \mc{O}_{\mathbb{P}^2 \times
\mathbb{P}^2}(-2,-2)^{\oplus 3} \ra \mc{O}_{\mathbb{P}^2
\times \mathbb{P}^2}(-1,-1)^{\oplus 3} \ra \mc{I}_{V_2}
\ra 0.
\]}
\noindent Twisting the above sequence with $\mc{O}_{\mathbb{P}^2 \times \mathbb{P}^2}(1,1)$ we get
{\small\[
0 \ra
\mc{O}_{\mathbb{P}^2 \times \mathbb{P}^2}(-2,-2) \ra
\mc{O}_{\mathbb{P}^2 \times \mathbb{P}^2}(-1,-1)^{\oplus 3} \ra
\mc{O}_{\mathbb{P}^2 \times \mathbb{P}^2}^{\oplus 3} \overset{f}{\ra}
\mc{I}_{V_2}(1,1) \ra
0.
\]}
\noindent Let $\mc{K}=\ker(f)$. Then
$h^0(\mc{I}_{V_2}(1,1)) = 3  - h^0(\mc{K}) + h^1(\mc{K})$. On the other hand, $H^1(\mc{O}_{\mathbb{P}^2}(j)) = H^2(\mc{O}_{\mathbb{P}^2}(j) = 0$ for $j=-1,-2$. Thus the K\"unneth formula applied the cohomology of the short exact sequence
$$0 \ra \mc{O}_{\mathbb{P}^2 \times \mathbb{P}^2}(-2,-2) \ra \mc{O}_{\mathbb{P}^2 \times \mathbb{P}^2}(-1,-1)^{\oplus 3} \ra \mc{K} \ra 0$$ results in $h^0(\mc{K}) = h^1(\mc{K}) = 0$. Hence $h^0(\mc{I}_{V_2}(1,1))=3$.

Now using the long exact sequence of cohomology arising from the short exact sequence 
$$0 \ra \mc{I}_{V_2}(1,1) \ra \mc{O}_{\mathbb{P}^2 \times \mathbb{P}^2}(1,1) \ra \mc{O}_{V_2}(1,1) \ra 0,$$
and the facts:
\[
\begin{array}{l}
h^0(\mc{I}_{V_2}(1,1)) = 3,\\  
h^0(\mc{O}_{\mathbb{P}^2 \times \mathbb{P}^2}(1,1)) = 9\\
h^0(\mc{O}_{V_2}(1,1)) = \dim_k B_2 = 6,\\
 h^1(\mc{O}_{\mathbb{P}^2 \times \mathbb{P}^2}(1,1)) = 0,
\end{array}
\]
we conclude that $H^1(\mc{I}_{V_2}(1,1)) = 0$.

For $d \geq 2$ we will construct a commutative diagram to assist with the study of the cohomology of the ideal sheaf $\mc{I}_{\frac{V_{d+1}}{V_d \times \mathbb{P}^2}}(\ones)$.
Recall from (\ref{eq:normseq}) that we have the following normalization sequence for $V_d$:
\[\hspace{1in} 0 \lra \mc{O}_{V_d} \lra \ds \bigoplus_{i=1}^6 \mc{O}_{W_{d,i}} \lra \ds \bigoplus_{k=1}^6 \mc{O}_{\{v_{d,k}\}} \lra 0.\hspace{.6in} (\dag_d)\]
Consider the sequence 
$$pr_{1,\dots,d}^{\ast} \left( (\dag_d) \ten \mc{O}_{(\mathbb{P}^2)^{\times d}}(\ones) \right) \ten_{\mc{O}_{(\mathbb{P}^2)^{\times d+1}}} pr_{d+1}^{\ast} \mc{O}_{\mathbb{P}^2}(1)$$ and its induced sequence of restrictions to $V_{d+1}$, namely
{\footnotesize \begin{equation} \label{eq:rest}
0 \ra \mc{O}_{V_d \times \mathbb{P}^2}(\ones)\big|_{V_{d+1}} \ra \ds \bigoplus_{i=1}^6 \mc{O}_{W_{d,i} \times \mathbb{P}^2}(\ones)\big|_{V_{d+1}}
\ra \ds \bigoplus_{k=1}^6\mc{O}_{\{v_{d,k}\}\times \mathbb{P}^2}(\ones) \big|_{V_{d+1}} \ra 0.
\end{equation}}

\noindent Now $V_{d+1} \subseteq V_d \times \mathbb{P}^2$ and $(W_{d,i}\times\mathbb{P}^2)\cap V_{d+1} = W_{d+1,i}$ due to Proposition \ref{prop:gammadparam} and Remark \ref{rmk:Sing(V_d)}. We also have that $(\{v_{d,k}\}\times \mathbb{P}^2)\cap V_{d+1} = \{v_{d+1,k}\}$ for all $i$,$k$. Therefore the sequence (\ref{eq:rest}) is equal to $(\dag_{d+1}) \ten \mc{O}_{(\mathbb{P}^2)^{\times d+1}}(\ones)$. In other words, we are given the commutative diagram:
\begin{figure}[h!]
\[
\xymatrix@-1pc{
0 \ar@{->}[r] & \mc{O}_{V_d \times \mathbb{P}^2}(\ones) \ar@{->}[dd] \ar@{->}[r]
& \ds \bigoplus_{i=1}^6 \mc{O}_{W_{d,i} \times \mathbb{P}^2}(\ones) \ar@{->}[dd] \ar@{->}[r]
& \ds \bigoplus_{k=1}^6 \mc{O}_{\{v_{d,k}\} \times \mathbb{P}^2}(\ones) \ar@{->}[dd] \ar@{->}[r]  &  0 \\\\
0 \ar@{->}[r] & \mc{O}_{V_{d+1}}(\ones)  \ar@{->}[r]
&  \ds \bigoplus_{i=1}^6 \mc{O}_{W_{d+1,i}}(\ones) \ar@{->}[r]
& \ds \bigoplus_{k=1}^6 \mc{O}_{\{v_{d+1,k}\}}(\ones) \ar@{->}[r] &   0.
}
\]
\vspace{.1in}
$$\text{Diagram 1: Understanding~}\mc{I}_{\frac{V_{d+1}}{V_d \times \mathbb{P}^2}}(1,\dots,1).$$
\end{figure}

\noindent where the vertical maps are given by restriction to $V_{d+1}$.
Observe that the kernels of the vertical maps (from left to right) are respectively
$\mc{I}_{\frac{V_{d+1}}{V_d \times \mathbb{P}^2}}(\ones)$, $\ds \bigoplus_i \mc{I}_{\frac{W_{d+1,i}}{W_{d,i} \times \mathbb{P}^2}}(\ones)$, and $\ds \bigoplus_k \mc{I}_{\frac{\{v_{d+1,k}\}}{\{v_{d,k}\} \times \mathbb{P}^2}}(\ones)$, and the cokernels are all 0. 

By the claim in $\S4.1$ and the K\"unneth formula, we have that $$H^1(\mc{O}_{V_d \times \mathbb{P}^2}(\ones)) = H^1(\mc{O}_{V_{d+1}}(\ones)) = 0.$$ Hence the application of the global section functor to Diagram 1 yields Diagram 2 below.  
Now by the Snake Lemma, we get the following sequence:
{\small $$\hspace{1.19in} \dots \lra \ds \bigoplus_{i=1}^6 H^0\left(\mc{I}_{\frac{W_{d+1,i}}{W_{d,i} \times \mathbb{P}^2}}(\ones)\right)
\overset{\psi}{\lra} \ds \bigoplus_{k=1}^6 H^0\left(\mc{I}_{\frac{\{v_{d+1,k}\}}{\{v_{d,k}\} \times \mathbb{P}^2}}(\ones)\right)$$
$$\lra H^1\left(\mc{I}_{\frac{V_{d+1}}{V_d \times \mathbb{P}^2}}(\ones)\right)
\lra \ds \bigoplus_{i=1}^6 H^1\left(\mc{I}_{\frac{W_{d+1,i}}{W_{d,i} \times \mathbb{P}^2}}(\ones)\right) \ra \dots .\hspace{1.22in}$$}
\smallskip

In Lemma \ref{lem:beta}, we will show that $\ds \bigoplus_{i} H^1\Big(\mc{I}_{\frac{W_{d+1,i}}{W_{d,i} \times \mathbb{P}^2}}(\ones)\Big) = 0$ for $d \geq 2$. Furthermore the surjectivity of the map $\psi$ will follow from Lemma \ref{lem:psisurj}. This will complete the proof of Proposition \ref{prop:H^1(IV)}.

\begin{landscape}
{\small \[
\xymatrix@-1pc{
& 0 \ar@{->}[d] && 0 \ar@{->}[d] &&  0 \ar@{->}[d] & \\
& H^0\left(V_d \times \mathbb{P}^2, \mc{I}_{\frac{V_{d+1}}{V_d \times \mathbb{P}^2}}(\ones)\right) \ar@{->}[dd]&&
\ds \bigoplus_{i=1}^6 H^0\left(W_{d,i} \times \mathbb{P}^2, \mc{I}_{\frac{W_{d+1,i}}{W_{d,i} \times \mathbb{P}^2}}(\ones)\right) \ar@{->}[dd]_{\alpha}  \ar@{-->}[rr]^{\psi} && \ds \bigoplus_{k=1}^6 H^0\left(\{v_{d,k}\} \times \mathbb{P}^2,\mc{I}_{\frac{\{v_{d+1,k}\}}{\{v_{d,k}\} \times \mathbb{P}^2}}(\ones)\right) \ar@{->}[dd]^{\lambda} &\\\\
0 \ar@{->}[r] & H^0\left(V_d \times \mathbb{P}^2, \mc{O}_{V_d \times \mathbb{P}^2}(\ones) \right) \ar@{->}[rr] \ar@{->}[dd]&&
 \ds \bigoplus_{i=1}^6 H^0\left(W_{d,i}\times \mathbb{P}^2, \mc{O}_{W_{d,i} \times \mathbb{P}^2}(\ones) \right) \ar@{->}[dd]_{\beta} \ar@{->}[rr]^{\nu} &&  \ds \bigoplus_{k=1}^6 H^0\left(\{v_{d,k}\}\times \mathbb{P}^2, \mc{O}_{\{v_{d,k}\} \times \mathbb{P}^2}(\ones) \right) \ar@{->}[dd]^{\gamma} \ar@{->}[r] & 0\\\\
0 \ar@{->}[r] & H^0\left(V_{d+1}, \mc{O}_{V_{d+1}}(\ones) \right) \ar@{->}[rr] \ar@{->}[dd]&&
 \ds \bigoplus_{i=1}^6 H^0\left(W_{d+1,i}, \mc{O}_{W_{d+1,i}}(\ones) \right) \ar@{->}[rr] \ar@{->}[dd] && \ds \bigoplus_{k=1}^6 H^0\left(\{v_{d+1,k}\}, \mc{O}_{\{v_{d+1,k}\}}(\ones) \right) \ar@{->}[dd] \ar@{->}[r] & 0\\\\
& H^1\left(V_d \times \mathbb{P}^2, \mc{I}_{\frac{V_{d+1}}{V_d \times \mathbb{P}^2}}(\ones)\right) && \ds \bigoplus_{i=1}^6 H^1\left(W_{d,i} \times \mathbb{P}^2, \mc{I}_{\frac{W_{d+1,i}}{W_{d,i} \times \mathbb{P}^2}}(\ones)\right) && \ds \bigoplus_{k=1}^6 H^1\left(\{v_{d,k}\} \times \mathbb{P}^2,\mc{I}_{\frac{\{v_{d+1,k}\}}{\{v_{d,k}\} \times \mathbb{P}^2}}(\ones)\right) &
}
\]}
\vspace{.4in}
\begin{center}{Diagram 2: Induced Cohomology from Diagram 1}\end{center}
\end{landscape}

\begin{lem} \label{lem:beta} $\ds \bigoplus_{i=1}^6 H^1\Big(W_{d,i} \times \mathbb{P}^2,~ \mc{I}_{\frac{W_{d+1,i}}{W_{d,i}\times \mathbb{P}^2}}(\ones)\Big) = 0$ for $d \geq 2$.
\end{lem}

\begin{proof} We consider the different parities of $d$ and $i$ separately.
For $d$ even and $i$ odd,
$$\mc{I}_{\frac{W_{d+1,i}}{W_{d,i} \times \mathbb{P}^2}} ~\cong~ \mc{O}_{W_{d,i} \times \mathbb{P}^2}(0,\dots,0,-1)$$
because $W_{d+1,i}$ is defined in $W_{d,i} \times \mathbb{P}^2$ by
one equation of degree $(0,\dots,0,1)$ (Proposition \ref{prop:gammadparam}).
Twisting by
$\mc{O}_{(\mathbb{P}^2)^{\times d+1}}(1,\dots,1)$ results in
\begin{equation} \label{eq:H^1(IVd)}
H^1\left(\mc{I}_{\frac{W_{d+1,i}}{W_{d,i}\times \mathbb{P}^2}}(1,\dots,1)\right) \cong H^1\left(\mc{O}_{W_{d,i} \times \mathbb{P}^2}(1,\dots,1,0)\right).
\end{equation}
Since $W_{d,i}$ is the product of $\mathbb{P}^1$ and points lying in $\mathbb{P}^2$ and  $H^1(\mc{O}_{\mathbb{P}^1}(1)) = H^1(\mc{O}_{\{pt\}}(1))= H^1(\mc{O}_{\mathbb{P}^2}) = 0$, the K\"unneth formula implies that the right hand side of  (\ref{eq:H^1(IVd)}) is equal to zero.

Consider the case of $d$ and $i$ even. As $pr_{1,\dots,d}(W_{d+1,i}) = W_{d,i}$ and \linebreak $pr_{d+1}(W_{d+1,i}) = [1:\omega:\omega^2]$ for $\omega = \omega_{d,i}$ a third of unity, we have that $W_{d+1,i}$ is defined in $W_{d,i} \times \mathbb{P}^2$ by two equations of degree $(0,\dots,0,1)$. The defining equations (in variables $x,y,z$) of $[1:\omega:\omega^2]$ form a $k[x,y,z]$-regular sequence and so we have the Koszul resolution of $\mc{I}_{\frac{W_{d+1,i}}{W_{d,i}\times \mathbb{P}^2}} \ten \mc{O}_{(\mathbb{P}^2)^{\times d+1}}(1,\dots,1)$:
\begin{equation} \label{eq:resI(Vd)}
\begin{aligned}
0 \ra \mc{O}_{W_{d,i} \times \mathbb{P}^2}(1,\dots,1,-1) &\ra
\mc{O}_{W_{d,i} \times \mathbb{P}^2}(1,\dots,1,0)^{\oplus 2}\\
&\ra \mc{I}_{\frac{W_{d+1,i}}{W_{d,i} \times \mathbb{P}^2}}(1, \dots, 1) ~~\ra~ 0.
\end{aligned}
\end{equation}

\noindent Now apply the global section functor to sequence (\ref{eq:resI(Vd)}) and note that $$H^j(\mc{O}_{W_{d,i}}(1,\dots,1)) = H^j(\mc{O}_{\mathbb{P}^2}) = H^j(\mc{O}_{\mathbb{P}^2}(-1))= 0 \text{~~~~for~} j=1,2.$$

\noindent Hence the K\"unneth formula yields
$$H^1\left(\mc{O}_{W_{d,i} \times \mathbb{P}^2}(1,\dots,1,0)\right)^{\oplus 2} =
H^2\left(\mc{O}_{W_{d,i} \times \mathbb{P}^2}(1,\dots,1,-1)\right) = 0.$$
Therefore $H^1\Big(\mc{I}_{\frac{W_{d+1,i}}{W_{d,i} \times \mathbb{P}^2}}(\ones)\Big) = 0$ for $d$ and $i$ even.

We conclude that for $d$ even, we know {\small $\ds \bigoplus_{i=1}^6 H^1\Big(\mc{I}_{\frac{W_{d+1,i}}{W_{d,i} \times \mathbb{P}^2}}(\ones) \Big)$} = 0.
For $d$ odd, the same conclusion is drawn by swapping the arguments for the $i$
even and $i$ odd subcases. 
\end{proof}

\begin{lem}\label{lem:psisurj}
The map $\psi$ is surjective for $d \geq 2$.
\end{lem}

\begin{proof} Refer to the notation from Diagram 2. To show $\psi$ is onto, here is our plan of attack.
\begin{enumerate}
\item Choose a basis of  {\small $\ds \bigoplus_k H^0\Big( \mc{I}_{\frac{\{v_{d+1,k}\}}{\{v_{d,k}\} \times \mathbb{P}^2}}(\ones)\Big)$} so that each basis element $t$ lies in {\small $H^0\Big( \mc{I}_{\frac{\{v_{d+1,k_0}\}}{\{v_{d,k_0}\} \times \mathbb{P}^2}}(\ones)\Big)$} for some $k=k_0$. For such a basis element $t$, identify its image under $\lambda$ in {\small $\ds \bigoplus_k H^0\left( \mc{O}_{\{v_{d,k}\} \times \mathbb{P}^2}(\ones) \right)$}.
\item Construct for $\lambda(t)$ a suitable preimage $s \in \nu^{-1}(\lm (t))$.
\item Prove $s \in$ ker($\beta$).
\end{enumerate}
As a consequence, $s$ lies in $\ds \bigoplus_i$ {\small $H^0\Big( \mc{I}_{\frac{W_{d+1,i}}{W_{d,i}\times \mathbb{P}^2}}(\ones)\Big)$} and serves as a preimage to $t$ under $\psi$. In other words, $\psi$ is surjective.
To begin, fix such a basis element $t$ and integer $k_0$.
\medskip

\noindent {\bf Step 1:}
Observe that $pr_{1,\dots,d}(\{v_{d+1,k_0}\}) = \{v_{d,k_0}\}$ and $pr_{d+1}(\{v_{d+1,k_0}\}) = [1:\omega:\omega^2]$ for some $\omega$, a third root of unity (Remark \ref{rmk:Sing(V_d)}). Thus our basis element $t \in$ {\small $\ds \bigoplus_k H^0\Big( \mc{I}_{\frac{\{v_{d+1,k}\}}{\{v_{d,k}\} \times \mathbb{P}^2}}(\ones)\Big)$} is of the form
\begin{equation}\label{eq:t}
t= a(\omega x_d - y_d) + b(\omega^2 x_d - z_d)
\end{equation}
\noindent for some $a,b \in k$, with \{$\omega x_d - y_d$, $\omega^2 x_d - z_d$\} defining
$[1:\omega:\omega^2]$ in the $(d+1)^{st}$ copy of $\mathbb{P}^2$.
Note that $\lambda$ is the inclusion map so we may refer to $\lambda(t)$ as $t$.
This concludes Step 1.\qed
\medskip

\noindent {\bf Step 2:} Next we construct a suitable preimage $s \in \nu^{-1}(\lambda(t))$. Referring to Remark \ref{rmk:Sing(V_d)}, let us observe
that for all $k$, there is an unique even integer := $i_k''$ and unique odd integer := $i_k'$ so that $v_{d,k} \in W_{d,i_k''} ~\cap~
W_{d,i_k'}$ for all $k= 1,\dots,6$.
For instance with $k_0=1$, we consider the membership $v_{d,1} \in W_{d,2} \cap
W_{d,3}$; hence $i_1'' = 2$ and $i_1' = 3$. 

As a consequence, $\lambda(t)$ has preimages under $\nu$ in
{\small
$$H^0\left(W_{d,i_{k_0}''} \times
\mathbb{P}^2, \mc{O}_{W_{d,i_{k_0}''} \times
\mathbb{P}^2}(\ones)\right)
\oplus 
~H^0\left(W_{d,i_{k_0}'} \times
\mathbb{P}^2, \mc{O}_{W_{d,i_{k_0}'} \times
\mathbb{P}^2}(\ones)\right).$$}

\vspace{-.1in}

\noindent For $d$ even (respectively odd) we write $i_{k_0} := i_{k_0}''$ (respectively $i_{k_0} := i_{k_0}'$). Therefore we intend to construct $s \in \nu^{-1}(t)$ belonging to $H^0\Big(\mc{O}_{W_{d,i_{k_0}} \times
\mathbb{P}^2}(\ones)\Big)$.
However this $W_{d,i_{k_0}}$ will also contain another point
$v_{d,j}$ for some $j \neq k_0$. Let us define the
global section $\tilde s \in H^0\Big(\mc{O}_{W_{d,i_{k_0}} \times
\mathbb{P}^2} (\ones) \Big)$ as follows. 
Since $\mc{O}_{W_{d,i_{k_0}}}(\ones)$ is a very ample sheaf, we have a global section $\tilde s_{k_0}$ separating the points $v_{d,k_0}$ and $v_{d,j}$; say $\tilde s_{k_0}(v_{d,k}) = \delta_{k_0, k}$.
We then use  ($\ref{eq:t}$) to define
$\tilde s$ by 
\[
\tilde s  =  \tilde s_{k_0} \cdot [a(\omega x_d - y_d) + b(\omega^2 x_d - z_d)].
\]
where $[1:\omega:\omega^2] = pr_{d+1}(\{v_{d+1,k_0}\})$.
We now extend this section $\tilde s$ to 
{\small $$s \in \ds \bigoplus_{i=1}^6 H^0 \left(\mc{O}_{W_{d,i} \times \mathbb{P}^2}(\ones) \right) \cong \left(\ds \bigoplus_{i=1}^6 H^0 \left(\mc{O}_{W_{d,i}}(\ones) \right) \right) \ten H^0(\mc{O}_{\mathbb{P}^2}(1)).$$}

\noindent This is achieved by setting
$s = \tilde s$ on $W_{d,i_{k_0}} \times \mathbb{P}^2$ and 0 elsewhere.
To check that $\nu(s) = t$, note
{\small 
\begin{equation}\label{eq:s}
\hspace{.5in} s = \bigoplus_{i=1}^6 s_i \text{~~where~}
s_i \in H^0 \left(\mc{O}_{W_{d,i} \times \mathbb{P}^2}(\ones) \right),~
s_i=
\begin{cases}
\tilde s, &i = i_{k_0},\\
0, & i \neq i_{k_0};
\end{cases}
\end{equation}}

Therefore by the construction of $\tilde s$, we have $\nu(\tilde s) = t \big |_{\{v_{d,k_0}\} \times \mathbb{P}^2}$. Hence we have built our desired preimage $s \in \nu^{-1}(t)$ and this concludes Step 2.\qed
\bigskip

\noindent {\bf Step 3:}
Recall the structure of $s$ from  (\ref{eq:s}). By definition of $\beta$, we have
that $\beta(s) = \beta \left( \bigoplus_{i=1}^6 s_i
\right)$ is  equal to
$\bigoplus_{i=1}^6 \left(s_i \big|_{W_{d+1,i}} \right)$.

For $i \neq i_{k_0}$, we clearly get that $s_i \big|_{W_{d+1,i}} = 0$. On the other hand, the key point of our construction is that
 $W_{d+1, i_{k_0}} = W_{d,i_{k_0}} \times [1:\epsilon:\epsilon^2]$ for some $\epsilon^3 =1$ as $i_{k_0}$ is chosen to be even (respectively odd) when $d$ is even (respectively odd) (Proposition \ref{prop:gammadparam}). Moreover $v_{d+1,k_0} \in W_{d+1,i_{k_0}}$ and $$pr_{d+1}(W_{d+1,i_{k_0}}) = pr_{d+1}(\{v_{d+1,k_0}\}) = [1:\omega:\omega^2]$$
where $\omega$ is defined by Step 1 and Remark \ref{rmk:Sing(V_d)}.  Thus $\epsilon = \omega$. Now we have
\[
s_{i_{k_0}} \big|_{W_{d+1,i_{k_0}}}  =
\tilde s_{k_0} \cdot [a(\omega x_d - y_d) + b(\omega^2 x_d - z_d)] \Big|_{[1:\omega:\omega^2]} = 0.
\]

\noindent Therefore $s_i \big|_{W_{d+1,i}} = 0$ for all $i=1, \dots,6$. Hence $\beta(s) =0$.\qed
\medskip

\noindent Hence Steps 1-3 are complete which concludes the proof of Lemma \ref{lem:psisurj}. \end{proof}

\noindent Consequently, we have verified Proposition \ref{prop:H^1(IV)}. \qed
\medskip

One of the main results why twisted homogeneous coordinate rings are so useful for studying Sklyanin algebras is that tcrs are factors of their corresponding Sklyanin algebra (by some homogeneous element; refer to Theorem \ref{thm:Skl3}). The following corollaries to Theorem \ref{thm:Bgenindeg1} illustrate an analogous result for $S_{deg}$.

\begin{cor} \label{cor:factor} Let $B$ be the point parameter ring of a degenerate Sklyanin algebra $S_{deg}$. Then $B \cong S_{deg}/K$ for some ideal $K$ of $S_{deg}$ that has six generators of degree 4 and possibly higher degree generators.
\end{cor}

\begin{proof}
By Theorem \ref{thm:Bgenindeg1}, $S_{deg}$ surjects onto $B$ say with kernel $K$. By Remark \ref{rmk:S(1bc)} we have that $\dim_k S_4 = 57$, yet we know $\dim_k B_4 = 63$ by Proposition \ref{prop:|B_d|}. Hence $\dim_k K_4 = 6$. The same results also imply that $\dim_k S_d = \dim_k B_d$ for $d \leq 3$.
\end{proof}

\begin{cor}\label{cor:restofB} The ring $B= B(S_{deg})$ is neither a domain or Koszul.
\end{cor}

\begin{proof}
By Corollary \ref{cor:S(1bc)}, there exist linear nonzero elements $u,v \in S$ with $uv =0$. The image of $u$ and $v$ are nonzero, hence $B$ is not a domain due to Corollary \ref{cor:factor}. Since $B$ has degree 4 relations, it does not possess the Koszul property.
\end{proof}


\begin{thebibliography}{abcd}

\bibitem{ASc}
M.~Artin and W.~Schelter, Graded algebras of global dimension 3, {\it Adv. in Math.} {\bf 66} (1987), 171--216.

\bibitem{ASta} M.~Artin and J.~T.~Stafford, Noncommutative graded domains with quadratic growth, {\it Invent. Math.} {\bf 122} (1995), 231--276.

\bibitem{ATV1}
M.~Artin, J.~Tate, and M.~van den Bergh, Some algebras associated to automorphisms of elliptic curves, {\it The Grothendieck Festschrift}, vol. 1, Birkh\"auser (1990), 33--85. 

\bibitem{ATV2}
M.~Artin, J.~Tate, and M.~van den Bergh, Modules over regular algebras of dimension 3, {\it Invent. Math.} {\bf 106} (1991), 335--389.

\bibitem{AZ}
M.~Artin and J.~J.~Zhang, Abstract Hilbert Schemes, {\it Alg. Rep. Theory} {\bf 4} (2001), 305--394.

\bibitem{Be}
G.~M.~Bergman, The Diamond Lemma for Ring Theory, {\it Adv. Math.} {\bf 29} (1978), 178--218.

\bibitem{BjSta}
J.~E.~Bj\"ork and J.~T.~Stafford, email correspondence with M. Artin, January 31, 2000.

\bibitem{BoG}
E.~Bombieri and W.~Gubler, {\it Heights in Diophantine Geometry}, Cambridge University Press, 2006.

\bibitem{CSh}
T.~Cassidy and B.~Shelton, Generalizing the Notion of Koszul Algebra, math.RA/0704.3752v1.

\bibitem{GW} K.~R.~Goodearl and R.~B.~Warfield, Jr., {\it An Introduction to Noncommutative Noetherian Rings}, London. Math. Soc. Student Texts vol. 61, Cambridge University Press, 2004.

\bibitem{Ha}
R.~Hartshorne, {\it Algebraic Geometry}, Graduate Text in Mathematics 52, Springer-Verlag, New York, 1977.

\bibitem{KeRSta} D.~S.~Keeler, D.~Rogalski, and J.~T.~Stafford, Na\"ive Noncommutative Blowing Up, {\it Duke Math. J.} {\bf 126} (3) (2005), 491--546.

\bibitem{Kr}
U.~Kr\"{a}hmer, Notes on Koszul Algebras, {\tt \scriptsize www.impan.gov.pl/ \hspace{-.1in} $\sim$ \hspace{-.1in} kraehmer/ \hspace{-.1in} connected.pdf}.

\bibitem{PP}
A.~Polishchuk and L.~Positselski, {\it Quadratic Algebras}, Amer. Math. Soc. University Lecture Series {\bf 37}, 2005.

\bibitem{RSta}
D.~Rogalski and J.~T.~Stafford, {\it A Class of Noncommutative Projective Surfaces}, arXiv:math/0612657v1.

\bibitem{RZ}
D.~Rogalski and J.~J.~Zhang, Canonical maps to twisted rings, {\it Math. Z.} {\bf 259} (2) (2008), 433--455.

\bibitem{ShT}
B.~Shelton and  C. Tingey, On Koszul algebras and a new construction of Artin– Schelter regular algebras, {\it J. Alg.} {\bf 241} (2001), 789–-798.

\bibitem{SmSta} S.~P.~Smith and J.~T.~Stafford, Regularity of the 4-dimensional Sklyanin algebra, {\it Compositio Math.} {\bf 83} (1992), 259--289.

\bibitem{Sta}
J.~T.~Stafford, Math 715: Noncommutative Projective Algebraic Geometry (course notes), University of Michigan, Winter 2007.

\bibitem{SteZ}
D.~R.~Stephenson and J.~Zhang, Growth of Graded Noetherian Rings, {\it Proc. Amer. Math. Soc.} {\bf 125} (1997), 1593--1605.

\bibitem{Zh}
J.~J.~Zhang, {\it Twisted Graded Algebras and Equivalences of Graded Categories}, Proc. London Math Soc. (3) {\bf 72} (1996) 281--311.

\end{thebibliography}
\end{document}